\documentclass{article}
\usepackage[latin1]{inputenc}
\usepackage{subfigure}
\usepackage[T1]{fontenc}
\usepackage{amsthm}
\usepackage{graphicx}
\usepackage{amsfonts}
\usepackage{amssymb}
\usepackage{amsmath, xypic}
\newtheorem{thm}{Theorem}

\newtheorem{lem}{Lemma}
\newtheorem{prop}{Proposition}
\newtheorem{cor}{Corollary}
\newtheorem{exm}{Example}

\newtheorem{rem}{Remark}

\newcommand{\al}{\alpha}

\newcommand{\la}{\lambda}
\newcommand{\eps}{\varepsilon}

\newlength{\cellsz}
\newcounter{cellsize}
\newcommand{\setcellsize}[1]{%
  \setcounter{cellsize}{#1}%
  \setlength{\cellsz}{\value{cellsize}\unitlength}}%
\setcellsize{11}%
\newcommand\cellify[1]{\def\thearg{#1}\def\nothing{}%
\hbox to 0pt{{\begin{picture}(\value{cellsize},\value{cellsize})
  \put(0,0){\line(1,0){\value{cellsize}}}
  \put(0,0){\line(0,1){\value{cellsize}}}
  \put(\value{cellsize},0){\line(0,1){\value{cellsize}}}
  \put(0,\value{cellsize}){\line(1,0){\value{cellsize}}} \end{picture}
\hss}}
\vbox to \cellsz{ \vss \hbox to \cellsz{\hss$#1$\hss} \vss}}
\newcommand\tableau[1]{\vcenter{\vbox{\let\\\cr
\baselineskip -16000pt \lineskiplimit 16000pt \lineskip 0pt
\ialign{&\cellify{##}\cr#1\crcr}}}}
\newcommand\tabl[1]{\vtop{\let\\\cr
\baselineskip -16000pt \lineskiplimit 16000pt \lineskip 0pt
\ialign{&\cellify{##}\cr#1\crcr}}}
\begin{document}

\title{Generating series formulas for the structure constants of Solomon's descent algebra}

\author{Alina R. Mayorova\and Ekaterina A. Vassilieva}



\maketitle
\abstract{Introduced by Solomon in his 1976 paper, the descent algebra of a finite Coxeter group received significant attention over the past decades. As proved by Gessel, in the case of the symmetric group its structure constants give the comultiplication table for the fundamental basis of quasisymmetric functions. We show that this property actually implies several well known relations linked to the Robinson-Schensted-Knuth correspondence and some of its generalisations. We further use the theory of type B quasisymmetric functions introduced by Chow to provide analogue results when the Coxeter group is the hyperoctahedral group.}

\section{Introduction}
For any integer $n$ denote $[n]=\{1,\cdots,n\}$, $S_n$ the symmetric group on $[n]$ and $id_n$ the identity permutation of $S_n$. A {\bf composition} $\al$ of $n$ denoted $\al \vDash n$ or $|\al| = n$ is a sequence of positive integers $\al=(\al_1,\al_2,\cdots,\al_p)$ of $\ell(\al)=p$ parts such that $\al_1 + \al_2 +\cdots = n$. If a composition $\la$ has its parts sorted in decreasing order, we say that $\la$ is a {\bf partition}  of $n$ and denote $\la \vdash n$. A partition $\la$ is usually represented as a~Young diagram of $n=|\la|$ boxes arranged in $\ell(\la)$ left justified rows so that the $i$-th row from the top contains $\la_i$ boxes. A Young diagram whose boxes are filled with the elements of $[n]$ such that the entries are strictly increasing along the rows and down the columns is called a {\bf standard Young tableau}. The partition given by the number of boxes in each row is its {\bf shape}.  As~an~example the following diagram is a standard Young tableau of shape $\la = (6,4,2,1,1)$.
\small
$$\tableau{  1& 3 & 4 & 8 & 9  & 14\\
 2 & 6 & 7 & 10\\
5 & 12  \\
11\\ 13 \\
}$$
\normalsize
Compositions of $n$ are in natural bijection with subsets of $[n-1]$. For $\al \vDash n$ and $I = \{i_1, i_2, \cdots, i_m\}$ a subset of $[n-1]$ such that $i_1 < i_2<\cdots i_m$, we denote $set(\al) = \{\al_1, \al_1 + \al_2, \cdots,\al_1 + \al_2 + \cdots \al_{p-1}\}$ and $comp(I) = (i_1, i_2-i_1,\cdots,i_m-i_{m-1}, n-i_m)$ this bijection and its inverse. As a result the number of compositions of $n$ is $2^{n-1}$.\\
One important feature of a permutation $\pi$ of $S_n$ is its {\bf descent set}, i.e the subset of $[n-1]$ defined by $set(\pi) = \{1\leq i \leq n-1\mid \pi(i)>\pi(i+1)\}$. For any subset $I \subseteq [n-1]$, denote $D_I$ (resp. $B_I$) the subset of $S_n$ containing all the permutations $\pi$ such that $set(\pi) = I$ (resp. $set(\pi) \subseteq I$). Similarly we define the {\bf descent set of a standard Young tableau} $T$ as the subset of $[n-1]$ defined by $set(T) = \{1\leq i \leq n-1\mid i $ is in a strictly higher row than $i+1\}$. For instance the descent set of the tableau in the previous example is $\{1,4,9,10,12\}.$ We denote $d_{\la I}$ the number of standard Young tableaux of shape $\la$ and descent set $I$.\\
By abuse of notation, we also denote $D_I$ (resp. $B_I$) the element of the algebra $\mathbb{C}S_n$ defined as the formal sum of the permutations it contains. In the more general context of finite Coxeter groups, Solomon showed in \cite{Sol76} that the $D_I$'s (resp. $B_I$'s) generate a subalgebra of $\mathbb{C}S_n$ of dimension $2^{n-1}$ usually referred to as the {\bf Solomon's descent algebra}. \\ More specifically, he showed that there exists a collection of (non-negative integral) {\bf structure constants} $(a^K_{IJ})_{I,J,K\subseteq[n-1]}$ and \linebreak $(b^K_{IJ})_{I,J,K\subseteq[n-1]}$ such that
\begin{align*}
D_ID_J = \sum_{K\subseteq[n-1]}a^K_{IJ}D_K,\;\;\;\;\;\;\;\;\; B_IB_J = \sum_{K\subseteq[n-1]}b^K_{IJ}B_K.
\end{align*}
As a result the number of ways to write a fixed permutation $\pi$ of $D_K$ as the ordered product of two permutations $\pi = \sigma\tau$ such that $\sigma \in D_I$ and $\tau \in D_J$ depends only on $set(\pi) = K$ and is equal to $a^K_{IJ}$. If we require instead $\sigma$ and $\tau$ to be respectively in $B_I$ and $B_J$ then the number of such products is $\sum_{K' \supseteq K} b_{I J}^{K'}$.
\begin{rem} The two families of structure constants are linked through the formula
\begin{equation}
\label{eq : ab}\sum_{I' \subseteq I, J' \subseteq J}a^K_{I'J'} = \sum_{K' \supseteq K} b_{I J}^{K'}. 
\end{equation}
\end{rem}
Because of its important combinatorial and algebraical properties the descent algebra received significant attention afterwards. In particular, Garsia and Reutenauer in \cite{GarReu89} provide a decomposition of its multiplicative structure and a new combinatorial interpretation of the coefficients $b^K_{IJ}$ in terms of the number of non-negative integer matrices with specified constraints. Bergeron and Bergeron give analogue results in \cite{BerBer92} when the symmetric group is replaced by the hyperoctahedral group (Coxeter group of type B). As shown by Norton in  \cite{Nor79}, the descent algebra of a finite Coxeter group is also strongly related to its {\it $\it 0$-Hecke algebra}. In particular, she proves that the dimension of each left principal indecomposable module of the $0$-Hecke algebra is equal to the cardinality of the analogue of one the $D_I$'s. She further shows that the analogues of the coefficients $a^K_{IJ}$ for $K = \emptyset$ are the entries of the {\it Cartan matrix} giving the number of times each irreducible module is a composition factor of each indecomposable module. In the specific case of the symmetric group $S_n$, Carter in \cite{Car86} uses the celebrated {\it Robison-Schensted (RS)} correspondence to explain the following relation obtained by computation of the Cartan matrix
\begin{equation}
\label{eq : car}a^\emptyset_{IJ} = \sum_{\la \vdash n}d_{\la I}d_{\la J}.
\end{equation}
Finally, Solomon's descent algebra is dual to the Hopf algebra of {\it quasisymmetric functions} (see Section~\ref{sec : typeA}) introduced by Gessel in \cite{Ges84}. In particular, he shows that the comultiplication table for their fundamental basis is given by the $a^K_{IJ}$'s.\\
Using the decomposition of Schur symmetric functions in terms of fundamental quasisymmetric functions one can easily show that Equation~(\ref{eq : car}) is a direct consequence of Gessel's result. In Section~\ref{sec : typeA}, we review this property and show that the consequences of Gessel's result also include a more general form of Equation~(\ref{eq : car}) relating the structure constants and the {\it Kronecker coefficients} as well as the equivalent of Equation~(\ref{eq : car}) for the generalisation of the RS correspondence provided by Knuth in \cite{Knu70} (RSK correspondence). Then we use the theory of quasisymmetric functions of type B introduced by Chow (\cite{Cho01}) in Section~\ref{sec : typeB} to show how these results transpose to the descent algebra of the hypercotahedral group. 

\section{Descent algebra of the symmetric group}
\label{sec : typeA}
\subsection{Quasisymmetric functions}
Let $X =\{x_1,x_2,\cdots,x_i,\cdots\}$ be a totally ordered set of commutative indeterminates. As introduced by Gessel in \cite{Ges84} a quasisymmetric function is a bounded degree formal power series in $\mathbb{C}[X]$ such that for any composition $(\alpha_1, \cdots \alpha_p)$ and any strictly increasing sequence of distinct indices $i_1 < i_2 < \dots < i_p$ the coefficient of $x_1^{\alpha_1}  x_2^{\alpha_2}  \cdots  x_p^{\alpha_p}$ is equal to the coefficient of $x_{i_1}^{\alpha_1}  x_{i_2}^{\alpha_2}  \cdots  x_{i_p}^{\alpha_p}$. In particular all symmetric functions on $X$ are quasisymmetric. Quasisymmetric functions are naturally indexed by compositions of $n$ and admit the monomial and fundamental quasisymmetric functions as classical bases.
$$M_{\alpha}(X) = \sum\limits_{i_1 < \cdots < i_p} x_{i_1}^{\alpha_1} x_{i_2}^{\alpha_2} \dots x_{i_p}^{\alpha_p},$$ $$F_{\alpha}(X) = \sum\limits_{\substack{i_1 \leq \cdots \leq i_n\\k\in set(\al) \Rightarrow i_{k+1} > i_k}} x_{i_1}x_{i_2} \cdots x_{i_n}.$$
These two bases are related through
\begin{equation}
\label{eq : FM} F_{\alpha}(X) = \sum_{set(\alpha) \subseteq set(\beta)\subseteq[n-1]} M_{\beta}(X).
\end{equation}
For $\la \vdash n$ denote also 
$p_\lambda(X)$ and $s_\lambda(X)$  the power sum and Schur symmetric functions indexed by partition $\la$.
The decomposition of Schur symmetric functions in the fundamental basis (see e.g. \cite[7.19.7]{Sta01}) is given by

\begin{equation}
\label{eq : SdF}s_\la(X) = \sum_{\al \vDash n}d_{\la set(\al)}F_\al(X).
\end{equation}

A {\bf semistandard Young tableau} is a Young diagram whose entries are strictly increasing down the column and non-decreasing along the rows. A tableau is said to have type $\mu = (\mu_1, \mu_2, \cdots)$ if it has $\mu_i$ entries equal to $i$. Denote $K_{\la \mu}$ the number of semistandard Young tableaux of shape $\la$ and type $\mu$. The decomposition of Schur symmetric functions in the monomial basis is

\begin{equation}
\label{eq : SKM}s_\la(X) = \sum_{\al \vDash n}K_{\la \al}M_\al(X).
\end{equation}

In what follows, using the natural correspondence between compositions and subsets we may denote $M_I$ and $F_I$ instead of $M_{comp(I)}$ and $F_{comp(I)}$ for $I \subseteq [n-1]$ and remove the reference to indeterminate $X$ when there is no confusion.
\subsection{Generating series and RSK-correspondence}
For two commutative sets of variables $X =\{x_1,x_2,\cdots,x_i,\cdots\}$ and $Y =\{y_1,y_2,\cdots,y_i,\cdots\}$, we denote $XY$ the set of indeterminates $\{(x_i, y_j); x_i \in X, y_i \in Y\}$ ordered by the lexicographical order. Gessel shows in \cite{Ges84} that for any subset $K \subseteq [n-1]$
\begin{equation}
\label{eq : Faff}F_K(XY) = \sum_{I,J \subseteq [n-1]}a_{IJ}^KF_I(X)F_J(Y).
\end{equation}
We can say that $F_K(XY)$ is the generating series for the coefficients $a_{IJ}^K$.\\
As stated in introduction, Equation (\ref{eq : car}) is a direct consequence of Equation (\ref{eq : Faff}). Indeed according to Equation~(\ref{eq : SdF}) $F_\emptyset = s_{(n)}$. Then using the {\it Cauchy identity} for Schur functions and applying Equation~(\ref{eq : SdF}) once again one gets:
\begin{align*}
\sum_{I,J \subseteq [n-1]}a_{IJ}^\emptyset F_I(X)F_J(Y) &=  F_\emptyset(XY)=s_{(n)}(XY),\\
&=\sum_{\substack{\la \vdash n\\ \phantom{I,J \subseteq [n-1]}}}s_\la(X)s_\la(Y),\\
&=\sum_{\substack{\la \vdash n\\ I,J \subseteq [n-1]}}d_{\la I}d_{\la J}F_I(X)F_J(Y).
\end{align*}
Denote $\la'$ the partition corresponding to the transposed Young diagram of $\la$. We have the following similar lemma:
\begin{lem}
For $I,J \subseteq [n-1]$, the numbers $a_{IJ}^{[n-1]}$ are given by
\begin{equation}
a_{IJ}^{[n-1]} = \sum_{\la \vdash n} d_{\la I}d_{\la' J}.
\end{equation}
\end{lem}
\begin{proof} According to Equation~(\ref{eq : SdF}) $F_{[n-1]} = s_{(1^n)}$. Then use the alternative Cauchy identity for Schur functions $$s_{(1^n)}(XY) = \sum_{\la \vdash n}s_\la(X)s_{\la'}(Y)$$ and apply Equation~(\ref{eq : SdF}) to $s_\la$ and $s_{\la'}$.\end{proof}
Generalising the RS-correspondence to matrices with non-negative integral entries, Knuth (\cite{Knu70}) proved that for $r, c \vDash n$ the number $m_{r, c}$ of such matrices with row and column sums equal respectively to $r$ and $c$ is given by
\begin{equation}
\label{eq : mKK} m_{r, c} = \sum_{\la \vdash n}K_{\la r}K_{\la c}.
\end{equation}
\begin{thm} \label{lem : mKK} Equation~(\ref{eq : mKK}) is also a consequence of Equation~(\ref{eq : Faff}). \end{thm}
\begin{proof} Denote $A_{I,J}^{K} = \sum_{I' \subseteq I,J' \subseteq J}a_{I',J'}^K$. Then according to Equation~(\ref{eq : FM})
\begin{align*}
\sum_{I',J' \subseteq [n-1]}\hspace{-4mm} a_{I'J'}^KF_{I'}(X)F_{J'}(Y) &= \sum_{I',J' \subseteq [n-1]}a_{I'J'}^K\sum_{I \supseteq I', J \supseteq J'}M_{I}(X)M_{J}(Y),\\
&= \sum_{I,J \subseteq [n-1]}M_{I}(X)M_{J}(Y)\sum_{I' \subseteq I, J' \subseteq J}a_{I'J'}^K,\\
& = \sum_{I,J \subseteq [n-1]}A_{I,J}^KM_{I}(X)M_{J}(Y).
\end{align*}
As a result, $A_{I,J}^\emptyset$ is the coefficient in $M_{I}(X)M_{J}(Y)$ of $F_{\emptyset}(XY)$ which we know is equal to $\sum_{\la \vdash n}s_\la(X)s_{\la}(Y)$ (see above). Finally use Equation~(\ref{eq : SKM}) to get
\begin{equation}
A_{I,J}^\emptyset = \sum_{\la \vdash n}K_{\la\, comp(I)}K_{\la\, comp(J)}.
\end{equation}
It remains to prove that $A_{I,J}^\emptyset  = m_{comp(I),comp(J)}$. The combinatorial interpretation of \cite{GarReu89} states that $b_{IJ}^K$ is the number of non-negative integer matrices $M$ with row and column sums equal to $comp(I)$ and $comp(J)$ respectively and with the word obtained by reading the entries of $M$ row by row from top to bottom equal to $comp(K)$ (zero entries being omitted). But according to Equation~(\ref{eq : ab}) $A_{I,J}^\emptyset = \sum_{K \subseteq [n-1]}b_{IJ}^K$. 
\end{proof}
%
\subsection{Relation to Kronecker coefficients}
We look at a more general version of Equations (\ref{eq : car}) and (\ref{eq : mKK}). Let $a_{I,J,K}^\emptyset = [D_{\emptyset}]D_ID_JD_K$ be the number of triples of permutations $\sigma_1$, $\sigma_2$, $\sigma_3$ of $S_n$ such that $set(\sigma_1) = I$, $set(\sigma_2) = J$, $set(\sigma_3) = K$ and $\sigma_1\sigma_2\sigma_3 = id_n$. Denote by $\chi^\la$ the irreducible character of $S_n$ indexed by partition $\la$.  For any $\la, \mu, \nu \vdash n$ define the {\bf Kronecker coefficient} $g(\la,\mu,\nu) = \frac{1}{n!}\sum_{\omega \in S_n}\chi^\la(\omega)\chi^\mu(\omega)\chi^\nu(\omega).$ We have the following result
\begin{thm}
\label{thm : gencartan}
Let $I,J$ and $K$ be subsets of $[n-1]$. The coefficient $a_{I,J,K}^\emptyset$ verifies
\begin{equation}
\label{eq : gencartan}a_{I,J,K}^\emptyset = \sum_{\la, \mu, \nu \vdash n}g(\la,\mu,\nu)d_{\la I}d_{\mu J}d_{\nu K}.
\end{equation}
\end{thm}
\begin{proof} The result of Gessel obviously extends to multi-index coefficients. In particular, the $a_{I,J,K}^{\emptyset}$'s verify
\begin{align*}
\sum_{I, J,K}a_{I,J,K}^{\emptyset}F_I(X)F_J(Y)F_K(Z) &= F_{\emptyset}(XYZ) = s_{(n)}(XYZ).
\end{align*}
Then use the generalised Cauchy identity (see \cite[I.7 Equation (7.9)]{Mac99}) for Schur functions to deduce that $$s_{(n)}(XYZ) = \sum_{\la, \mu, \nu}g(\la,\mu,\nu)s_\la(X)s_\mu(Y)s_\nu(Z)$$ and apply Equation~(\ref{eq : SdF}) to get Equation~(\ref{eq : gencartan}). 
\end{proof}

For $p, q, r \vDash n$ denote $m_{p,q,r}$ the number of three-dimensional arrays $M = (M_{i,j,k})$ with non-negative integer entries such that $p_k = \sum_{i,j}M_{i,j,k}$, $q_j =\sum_{i,k}M_{i,j,k}$ and $r_i =\sum_{j,k}M_{i,j,k}$. Similarly as Theorem~\ref{lem : mKK}, we have the following corollary to Theorem \ref{thm : gencartan}.
\begin{cor}
Let $p, q$ and $r$ be three compositions of integer $n$. The quantity $m_{p,q,r}$ defined above is given by
\begin{equation}
\label{eq  : mKKK} m_{p,q,r} = \sum_{\la, \mu, \nu \vdash n}g(\la,\mu,\nu)K_{\la p}K_{\mu q}K_{\nu r}.
\end{equation}
\end{cor}
A bijective proof of Equation(\ref{eq : mKKK}) is provided in \cite{AveVal12}. The proof involves semistandard Young tableaux and {\it Littlewood-Richardson tableaux}. 

\section{Descent algebra of the hyperoctahedral group}
\label{sec : typeB}
\subsection{Signed permutations and domino tableaux}
Let $B_n$ be the {\bf hyperoctahedral group} of order $n$. $B_n$ is composed of all permutations $\pi$ on the set $\{\text{-}n, \cdots,\text{-}2, \text{-}1, 0, 1, 2, \cdots, n\}$ such that for all $i$ in $\{0\}\cup[n]$, $\pi(-i) = -\pi(i)$ (in particular $\pi(0) = 0$). As a result, such permutations usually referred to as {\bf signed permutations} are entirely described by their restriction to $[n]$. The {\bf descent set} of a signed permutation $\pi$ of $B_n$ is the subset of $\{0\}\cup[n-1]$ defined by $set(\pi) = \{0 \leq i \leq n-1 \mid \pi(i) > \pi(i+1)\}$. The main difference with respect to the case of the symmetric group is the possible descent in position $0$ when $\pi(1)$ is a negative integer. There is a natural analogue of the RSK-correspondence for signed permutations involving {\bf domino tableaux}. For $\la \vdash 2n$, a {\bf standard domino tableau} $T$ of shape $sh(T) = \la$ is a semistandard Young tableau filled with the elements of $[n]$ such that each integer is used exactly twice and that for all $i$ in $[n]$, the two occurrences of $i$ are adjacent to each other. Merge the two squares labelled $i$ to get a rectangle with a single label $i$. We name such a rectangle a {\bf domino}. A domino can be either vertical (if the two $i$'s were on the same column) or horizontal. Note that not all the shapes $\la \vdash 2n$ may be entirely tiled by dominos (for example $\la =(3,2,1)$ may not). In the sequel we consider only the set $\mathcal{P}^0(n)$ of {\it empty $2$-core partitions} $\la \vdash 2n$ that fit such a tiling. Barbash and Vogan (\cite{BarVog82}) built a bijection between signed permutations of $B_n$ and pairs of standard domino tableaux of equal shape in $\mathcal{P}^0(n)$.\\
 A standard domino tableau $T$ has a descent in position $i>0$ if $i+1$ lies strictly below $i$ in $T$ and has descent in position $0$ if the domino filled with $1$ is vertical. We denote $set(T)$ the set of all its descents. For $\la$ in $\mathcal{P}^0(n)$ and $I \subseteq \{0\}\cup[n-1]$, $d^B_{\la I}$ is the number of standard domino tableaux of shape $\la$ and descent set $I$.
\begin{exm}
The domino in Figure~(\ref{fig : dom}) has shape $(5,5,4,1,1)$ and descent set \{0,3,5,6\}.
\begin{figure} [h]
\begin{center}
\includegraphics[width=2cm]{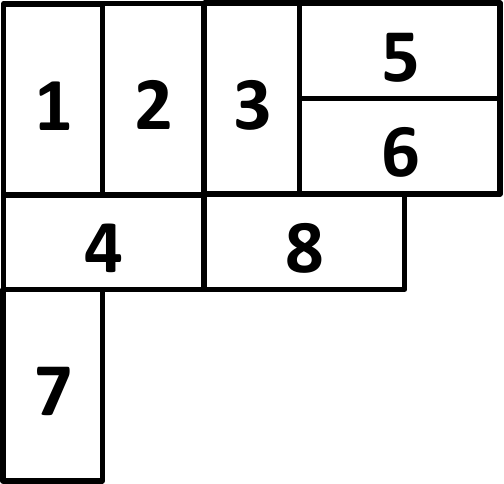}
\caption{A standard domino tableau}
\label{fig : dom}
\end{center}
\end{figure}
\end{exm}
Ta\c{s}kin (\cite{Tas12}) showed that the two standard domino tableaux associated to a signed permutation $\pi$ by the algorithm of Barbash and Vogan have respective descent sets $set(\pi^{-1})$ and $set(\pi)$. As a result, if  for any $I\subseteq \{0\}\cup[n-1]$ we denote $D^B_I$ the formal sum in $\mathbb{C}B_n$ of the signed permutations $\pi$ with $set(\pi) =I$ and  $c^K_{I_1,I_2,\cdots,I_p} = [D^B_K]\prod_jD^B_{I_j}$ the structure constants of Solomon's descent algebra of type B for $I_1,I_2,\cdots,I_p,K\subseteq \{0\}\cup[n-1]$, we have the following analogue of Equation~(\ref{eq : car}).

\begin{equation}
\label{eq : cdd} c^\emptyset_{IJ} = \sum_{\la \in \mathcal{P}^0(n)}d^B_{\la I}d^B_{\la J}.
\end{equation}

\subsection{Quasisymmetric functions of type B and main results}

Chow defines in his PhD thesis (\cite{Cho01}) an analogue of Gessel's algebra of quasisymmetric functions for signed permutations that is dual to Solomon's descent algebra of type B\footnote{Prior to Chow, Poirier (\cite{Poi98}) introduced an alternative type B analogue of the quasisymmetric functions used for example in \cite{AdiAthEliRoi15}. However Poirier's quasisymmetric functions are dual to the Mantaci-Reutenauer algebra and not Solomon's descent algebra of type B. See \cite{Pet05} for further details.}. \\ Let $X = \{\ldots,x_{-i},\ldots, x_{-2}, x_{-1}, x_0, x_1, x_2,\ldots, x_i,\ldots\}$ be a set of totally ordered commutative indeterminates with the further assumption that $x_{-i} = x_{i}$ and $I$ be a subset of $\{0\} \cup [n-1]$, he defines:
\begin{align*}
F_I^B(X) = \sum_{\substack{0\leq i_1\leq i_2\leq \ldots \leq i_n\\ j\in I \Rightarrow i_{j+1} > i_j}} x_{i_1}x_{i_2}\ldots x_{i_n},\\
M_I^B(X) = \sum_{\substack{j\notin I \Rightarrow i_{j+1} = i_j\\ j\in I \Rightarrow  i_{j+1} > i_j }} x_{i_1}x_{i_2}\ldots x_{i_n},
\end{align*}
where $i_0 = 0$. Note the particular rôle played by variable $x_0$.
\begin{exm}
Let $n=2$ and $X= \{x_{-2}, x_{-1}, x_0, x_1, x_2\}$ then $F_{\emptyset} =x_0^2 + x_1^2 + x_2^2 + x_0x_1+x_0x_2+x_1x_2$, $F_{\{1\}} = x_0x_1 +x_0x_2+ x_1x_2$, $F_{\{0\}} = x_1^2 + x_2^2 +x_1x_2$ and $F_{\{0,1\}} = x_1x_2$.
\end{exm}
Let $Y$ be a second copy of $X$ such that $X$ and $Y$ commute between each other. Chow uses the theory of $P$-partitions of type B to show that for $K\subseteq \{0\}\cup[n-1]$
\begin{align}
\label{eq : FcFF}
F^B_K(XY) = \sum_{I,J \subseteq \{0\}\cup[n-1]}c^K_{I,J}F^B_I(X)F^B_J(Y).
\end{align}
\begin{thm}
\label{thm : Bcdd}
Equation~(\ref{eq : cdd}) is a consequence of Equation~(\ref{eq : FcFF}).
\end{thm}
The irreducible characters of $B_n$ are naturally indexed by partitions of $\mathcal{P}^0(n)$. Denote $\psi^{\la}$ the character indexed by $\la \in \mathcal{P}^0(n)$ and define the {\bf Kronecker coefficients of the hyperoctahedral group} $g^B(\la,\mu,\nu) = \frac{1}{|B_n|}\sum_{\omega \in B_n} \psi^{\la}(\omega)\psi^{\mu}(\omega) \psi^{\nu}(\omega)$. A more general form of Theorem~\ref{thm : Bcdd} can be stated as follows.
\begin{thm}
\label{thm : Bcgddd}
Let $I,J$ and $K$ $\subseteq \{0\}\cup[n-1]$. The number $c_{I,J,K}^\emptyset$  of triples of signed permutations $\pi_1, \pi_2, \pi_3$ in $B_n$ such that $set(\pi_1) = I$, $set(\pi_2) = J$, $set(\pi_3) = K$ and $\pi_1\pi_2\pi_3 = id_{2n}$ verifies
\begin{align}
c^\emptyset_{I,J,K} = \sum_{\la, \mu, \nu \in \mathcal{P}^0(n)}g^B(\la,\mu,\nu)d^B_{\la I}d^B_{\mu J}d^B_{\nu K}.
\end{align}
\end{thm}
The proof of Theorems~\ref{thm : Bcdd}~and~\ref{thm : Bcgddd} uses Equation~(\ref{eq : FcFF}) and generating functions for domino tableaux. We detail it in the two following sections.
\begin{rem}Using the combinatorial interpretation of type B structure constants in \cite{BerBer92}, one can find an analogue of Equations~(\ref{eq : mKK}) and (\ref{eq : mKKK}) as a corollary to Theorem~\ref{thm : Bcgddd}. \end{rem}
\subsection{Domino functions and $2$-quotient}
Generating functions for domino tableaux sometimes called {\it domino functions} are well studied objects (see e.g. \cite{KirLasLecThi94}). We introduce a modified definition to get an analogue of Schur functions verifying Propositions \ref{thm : GdF} and \ref{prop : Gss} that we need to prove Theorem \ref{thm : Bcgddd}. Our development differs by the addition of '$0$' entries to the domino tableaux in some cases.\\
A {\bf semistandard domino tableau} of shape $\la \in \mathcal{P}^0(n)$ and weight $\mu =(\mu_0$, $\mu_1$, $\mu_2$, $\cdots)$ with $\mu_i \geq 0$ and $\sum_i \mu_i = n$ is a tiling of the Young diagram of shape $\la$ with horizontal and vertical dominos labelled with integers in $\{0,1,2,\cdots\}$ such that labels are non decreasing along the rows, strictly increasing down the columns and exactly $\mu_i$ dominos are labelled with $i$. If the top leftmost domino is vertical, it cannot be labelled $0$ (we leave it to the reader to check that the only possible sub-pattern of dominos with label $0$ in a semistandard domino tableau is a row composed of horizontal dominos).  
The {\bf standardisation} $T^{st}$ of a semistandard domino tableau $T$ of weight $\mu$ is the standard domino tableau obtained by relabelling the dominos of $T$ with $1,2,\cdots,n$ such that the dominos labeled with $i_m = \min\{i\mid\mu_i>0\}$ are relabelled with $1, 2, \cdots, \mu_{i_m}$ from left to right and so on. The following lemma is left to the reader. 
\begin{lem}
\label{lem : std}
Let $T$ be a semistandard domino tableau of shape $\la \in \mathcal{P}^0(n)$ and weight $\mu$ and let $T^{st}$ be its standardisation. Denote $\overline{\mu}$ the composition of $n$ obtained by removing the $0$'s in $\mu$ and let $S(T)$ be the union of $set(\overline{\mu})$ and possibly $\{0\}$ if the top leftmost domino is vertical. We have  
\begin{equation}
set(T^{st}) \subseteq S(T).
\end{equation}
Conversely given a standard domino tableau $T_0$, and $\mu =(\mu_0$, $\mu_1$, $\mu_2,$ $\cdots)$ such that $\mu_0 = 0$ if the top leftmost domino of $T_0$ is vertical and $set(T_0)\setminus\{0\} \subseteq set(\overline{\mu})$ there exists exactly one semistandard domino tableau $T$ of weight $\mu$ with $T^{st} = T_0$.  
\end{lem}
Proceed with the definition of the domino functions. 
Given indeterminate $X$ and a semistandard domino tableau $T$ of weight $\mu$, denote $X^T$ the monomial $x_0^{\mu_0}x_1^{\mu_1}x_2^{\mu_2}\cdots$. For $\la \in \mathcal{P}^0(n)$ we call the {\bf domino function} indexed by $\la$ the function $\mathcal{G}_\la$ defined as 
\begin{align}
\label{eq : defG}
\mathcal{G}_\la(X)  &= \sum_{sh(T) =\la}{X}^{T}
\end{align}
where the sum is on all semistandard domino tableaux of shape $\la$. The following proposition links domino functions and type B quasisymmetric functions. 
\begin{prop}
\label{thm : GdF}
For $\la \in \mathcal{P}^0(n)$, denote $\mathcal{G}_\la$ the domino function indexed by $\la$ defined in Equation~(\ref{eq : defG}). We have
\begin{equation}
\mathcal{G}_\la  = \sum_{I \subseteq \{0\}\cup[n-1]} d^B_{\la I} F^B_I.
\end{equation}
\end{prop}
\begin{proof}
From the definition of Equation~(\ref{eq : defG}), $\mathcal{G}_\la(X)$ contains exactly one monomial of the form $x_0^{\mu_0}x_1^{\mu_1}x_2^{\mu_2}\cdots$ for any semistandard domino tableau $T$ of weight $\mu$ with $\mu_0 = 0$ if the top leftmost domino of $T$ is vertical. Classify the set of semistandard domino tableaux according to their standardisation and add the monomials corresponding to the tableaux in the same class. According to Lemma~\ref{lem : std}, for each standard domino tableau $T_0$, the sum of the monomials corresponding to the semistandard domino tableaux $T$ such that $T^{st} = T_0$ is $\sum_{\substack{i_1 \leq i_2 \leq \cdots \leq i_n\\j\in set(T_0) \Rightarrow i_{j+1} > i_j }}x_{i_1}x_{i_2}\cdots x_{i_n}$. As a result $\mathcal{G}_\la(X)  = \sum_{T_0} F^B_{set(T_0)}(X)$. 
\end{proof}
There is a well known bijection between semistandard domino tableaux of weight $\mu$ and {\bf bi-tableaux}, i.e. pairs of semistandard Young tableaux of respective weights $\mu^-$ and $\mu^+$ such that $\mu^+_i + \mu^-_i = \mu_i$ for all $i$. The respective shapes of the two Young tableaux depend only on the shape of the initial semistandard domino tableau. Denote $(T^-,T^+)$ the bi-tableau associated to a semistandard domino tableau $T$ of shape $\la$ and $(\la^-,\la^+)$ the shapes of $T^-$ and $T^+$. $(T^-,T^+)$ (resp. $(\la^-,\la^+)$) is the {\bf $2$-quotient} of $T$ (resp. $\la$). One can build $T^-$ and $T^+$ by filling each box of $T$ (a domino is composed of two boxes) by a '-' or a '+' sign such that the top leftmost box is filled with '-' and two adjacent boxes have opposite signs. $T^-$ (resp. $T^+$) is obtained from the sub-tableau of $T$ composed of the dominos with top rightmost box filled with '-' (resp. '+'). 
\begin{rem}
\label{rem : zero}
If the top leftmost domino of a semistandard domino tableau is vertical (rep. horizontal), its label is used for $T^-$ (resp. $T^+$). Therefore only $T^+$ may have entries equal to $0$. 
\end{rem}   
\begin{exm} Figure \ref{fig : SSDT} shows a domino tableau of weight \linebreak $\mu=(2,0,2,0,0,4,0,1)$, its standardisation and its 2-quotient. 
\begin{figure} [h]
\begin{center}
\includegraphics[width=10cm]{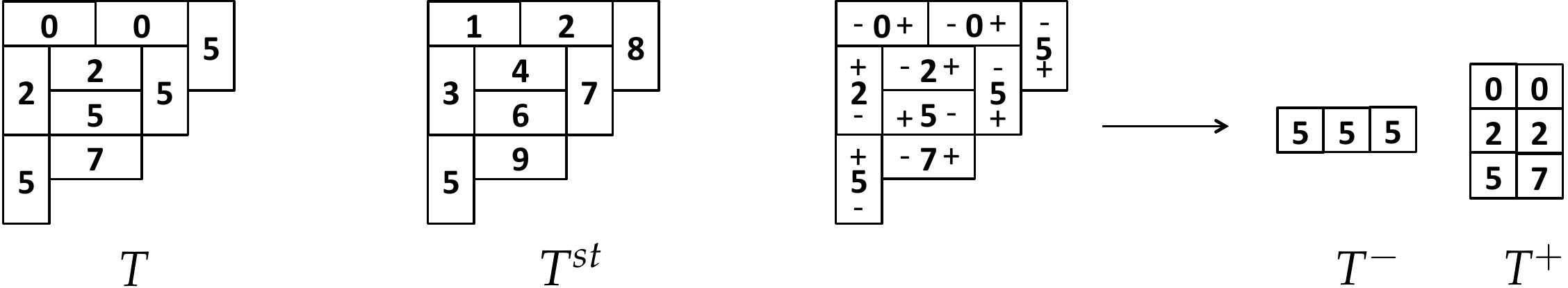}
\caption{A semistandard domino tableau, its standardisation and 2-quotient.}
\label{fig : SSDT}
\end{center}
\end{figure}
\end{exm}

Let $X^- = \{x_{-i}\}_{i>0}$ and $X^+ = \{x_{i}\}_{i\geq0}$ (note that because $x_{-i} = x_i$, $X^- = X^+\setminus\{x_0\}$). We have the following proposition.
\begin{prop}
\label{prop : Gss}
For $\la \in \mathcal{P}^0(n)$, the domino function $\mathcal{G}_\la$ and the Schur symmetric functions are related through
\begin{align}
\mathcal{G}_\la(X)= s_{\la^-}(X^-)s_{\la^+}(X^+).
\end{align}
\end{prop}
\begin{proof}
According to the definition of the $2$-quotient above and Remark~\ref{rem : zero}, one can rewrite Equation~(\ref{eq : defG}) as
\begin{align*}
\mathcal{G}_\la(X)  &= \sum_{sh(T^-) = \la^-,\;sh(T^+) = \la^+}{X^-}^{T^-}{X^+}^{T^+},
\end{align*}
where the sum is on all pairs of semistandard Young tableaux $T^-$ and $T^+$ of shape $\la^-$ and $\la^+$ such that no entry of $T^-$ is equal to $0$.
\end{proof}

\subsection{Domino functions on products of indeterminates}
We are now ready to give the proof of Theorem~\ref{thm : Bcdd}. Although it can be deduced directly from Theorem~\ref{thm : Bcgddd}, a direct proof using classical properties of Schur functions is of interest. According to Proposition~\ref{thm : GdF}, for sets of indeterminates $X$ and $Y$, $F^B_\emptyset(XY) = \mathcal{G}_{(2n)}(XY)$. Further, note that
\begin{align*}
(XY)^+ &= \{(x_i,y_j)\}_{(i,j)\geq (0,0)} = \{(x_0,y_j)\}_{j\geq 0}\cup \{(x_i,y_j)\}_{i > 0, j \in \mathbb{Z}},\\
&= X^+Y^+\cup X^-Y^-.
\end{align*}
As a consequence, using Proposition~\ref{prop : Gss} one has
\begin{align}
\nonumber \mathcal{G}_{(2n)}(XY) &= s_{(n)} ((XY)^+) = s_{(n)}(X^-Y^-\cup X^+Y^+),\\
\nonumber & =\sum_{k = 0}^n s_{(k)}(X^-Y^-)s_{(n-k)}(X^+Y^+),\\
\nonumber & = \sum_{k = 0}^n\sum_{\la^- \vdash k}s_{\la^-}(X^-)s_{\la^-}(Y^-)\sum_{{\la^+} \vdash n-k}s_{\la^+}(X^+)s_{\la^+}(Y^+),\\
\nonumber &= \sum_{\la \in \mathcal{P}^0(n)}s_{\la^-}(X^-)s_{\la^+}(X^+)s_{\la^-}(Y^-)s_{\la^+}(Y^+),\\
\label{eq : GGG}&=\sum_{\la \in \mathcal{P}^0(n)}\mathcal{G}_{\la}(X)\mathcal{G}_{\la}(Y). 
\end{align}
As a result $\sum_{I,J}c^\emptyset_{IJ}F^B_I(X)F^B_J(Y) = \sum_{\la \in \mathcal{P}^0(n)}\mathcal{G}_{\la}(X)\mathcal{G}_{\la}(Y)$. Finally use Proposition~\ref{thm : GdF} to get Equation~(\ref{eq : cdd}).\\
We proceed with the proof of Theorem~\ref{thm : Bcgddd}. For sets of indeterminates $X,Y$ and $Z$, Equation~(\ref{eq : FcFF}) implies that $$\sum_{I,J,K}c^\emptyset_{I,J,K}F_I(X)F_J(Y)F_K(Z) = F_{\emptyset}(XYZ).$$ But according to Proposition~\ref{thm : GdF}, $F_{\emptyset}(XYZ) = \mathcal{G}_{(2n)}(XYZ)$. We use the theory of symmetric functions on wreath products. In particular define as in \cite{AdiAthEliRoi15} for any set of indeterminates $U$ and $V$ and any partition $\eps$ of $n$:
\begin{align*}
p^+_\eps(U,V) = \prod_i \left [p_{\eps_i}(U)+ p_{\eps_i}(V)\right ], \;\;\;&\;\;\; p^-_\eps(U,V) = \prod_i\left [p_{\eps_i}(U)- p_{\eps_i}(V)\right ].
\end{align*}
Note that $p^+_\eps(U,V) = p_\eps(U\cup V)$ (as a multiset) and if $V\subseteq U$, $p^-_\eps(U,V) = p_\eps(U\setminus V)$. Then (see \cite[I, Appendix B]{Mac99}), for partition $\mu \in \mathcal{P}^{0}(n)$ of 2-quotient $(\mu^-,\mu^+)$,
\begin{align*}
p^-_{\mu^-}(U,V)p^+_{\mu^+}(U,V) = \sum_{\la \in \mathcal{P}^0(n)}\psi^{\la}_{\mu}s_{\la^-}(V)s_{\la^+}(U).
\end{align*}
Set $U =X^+$ and $V =X^-$ and use Proposition~\ref{prop : Gss} to get for any $\mu \in \mathcal{P}^{0}(n)$ 
$$p_{\mu^+}(X)(x_0)^{n-|\mu^+|} = \sum_{\la \in \mathcal{P}^0(n)}\psi^{\la}_{\mu}\mathcal{G}_{\la}(X).$$
But $p_{\mu^+}(XY)(x_0y_0)^{n-|\mu^+|} = p_{\mu^+}(X)(x_0)^{n-|\mu^+|}p_{\mu^+}(Y)(y_0)^{n-|\mu^+|}$. We get
\begin{align*}
\sum_{\la \in \mathcal{P}^0(n)} \psi^{\la}\mathcal{G}_{\la}(XY)= \sum_{\mu, \nu \in \mathcal{P}^0(n)}\psi^\mu\psi^\nu\mathcal{G}_{\mu}(X)\mathcal{G}_{\nu}(Y).
\end{align*}
As a result
\begin{align}
\label{eq : GgGG}\mathcal{G}_{\la}(XY)= \sum_{\mu, \nu \in \mathcal{P}^0(n)}g^B(\la,\mu,\nu)\mathcal{G}_{\mu}(X)\mathcal{G}_{\nu}(Y).
\end{align}
Replacing $Y$ by the product $YZ$ in Equation~(\ref{eq : GGG}) and applying Equation~(\ref{eq : GgGG}), we deduce that $$\mathcal{G}_{(2n)}(XYZ)= \sum_{\la, \mu, \nu \in \mathcal{P}^0(n)}g^B(\la,\mu,\nu)\mathcal{G}_{\la}(X)\mathcal{G}_{\mu}(Y)\mathcal{G}_{\nu}(Z).$$ One ends the proof of Theorem~\ref{thm : Bcgddd} by applying proposition~\ref{thm : GdF}.

\bibliography{biblio.bib} 
\bibliographystyle{abbrv}


\end{document}